\documentclass[10pt]{amsart}

\usepackage[utf8]{inputenc}
\usepackage{subfigure}
\usepackage[round]{natbib}
\usepackage{subfigure}
\usepackage{xspace}
\usepackage{graphicx}

\usepackage{amsmath}
\usepackage[foot]{amsaddr}
\usepackage{mathabx}
\usepackage{url}

\newtheorem{thm}{Theorem}
\newtheorem{prop}[thm]{Proposition}

\newtheorem{cor}[thm]{Corollary}
\newtheorem{Def}{Definition}

\newcommand\HH[3]{\ensuremath{H_{#2;#3}}}

\newcommand{\naturals}{\ensuremath{\mathbb N}}
\newcommand{\integers}{\ensuremath{\mathbb Z}}

\title[Constellations and multicontinued fractions]{Constellations and multicontinued
  fractions: application to Eulerian triangulations}

\author[M.~Albenque]{Marie Albenque$^1$}
\address{$^1$CNRS, LIX, École Polytechnique, 91128 Palaiseau Cedex, France}
\author[J.~Bouttier]{Jérémie Bouttier$^2$}
\address{$^2$Institut de Physique Théorique, CEA, IPhT, F-91191
  Gif-sur-Yvette Cedex, France, CNRS URA 2306 and 
  LIAFA, Université Paris Diderot - Paris 7, case 7014, 75205 Paris
  Cedex 13, France, CNRS UMR 7089}

\begin{document}
\begin{abstract}
We consider the problem of enumerating planar constellations with two
points at a prescribed distance. Our approach relies on a
combinatorial correspondence between this family of constellations and
the simpler family of rooted constellations, which we may formulate
algebraically in terms of multicontinued fractions and generalized
Hankel determinants.  As an application, we provide a combinatorial
derivation of the generating function of Eulerian triangulations with
two points at a prescribed distance.

\end{abstract}
\maketitle

\section{Introduction}
\label{sec:intro}
From the initial work of~\cite{Hurwitz} about transitive ordered
factorizations of the identity in the symmetric group, to the
bijective approach of \cite{BouSch00} and \cite{BDG04}, including the
more algebraic approach of~\cite{GouJack97}, constellations appear in
many forms in different areas of combinatorics. We refer the reader to
the book of~\cite{LandoZvonkin} for an extensive review of the variety
of contexts in which constellations appear.

In this paper, we focus on the map point of view, and consider the
problem of enumerating planar constellations with two points at a
prescribed distance. In~\cite{BDG04}, this problem has already been
tackled by giving a bijection between this family of constellations
and a family of decorated trees called mobiles, which results in
recurrence equations that characterize the associate generating
functions. Remarkably, these recurrence equations admit explicit
solutions~\cite[Section 6.7]{diFraIntegrable}, reminiscent of
integrable systems (see also \cite{BDG03a} for the solutions of the
equations corresponding to $2$-constellations). The motivation of our
work is to explain combinatorially the form of these
solutions. Following the approach of~\cite{BouGui10}, we exhibit a
connection between our enumeration problem and the \emph{a priori}
simpler problem of counting rooted constellations. This connection
relies on a decomposition involving some lattice paths, which we may
rephrase in terms of ``multicontinued fractions'', in the spirit
of~\cite{FlaContFrac}. Solving our enumeration problem amounts to
finding unknown coefficients in such a multicontinued fraction, which
can be done via a generalization of Hankel determinants. We illustrate
this approach in the case of Eulerian triangulations, where we provide
a combinatorial derivation of a formula found in \cite{BDG03b} for the
generating function of Eulerian triangulations with two points at
prescribed distance. These objects are in bijection with
``very-well-labeled trees with small labels'', see also \cite[Section
2.1]{MBM06}. The next section presents in more detail our main results
and the organization of the paper.

\section{Definitions and main results}
\label{sec:mainresults}

A \emph{planar map} is a proper embedding of a connected graph into the
sphere, where proper means that edges are smooth simple arcs which
meet only at their endpoints. Two planar maps are identical if one of
them can be mapped onto the other by a homeomorphism of the sphere
preserving the orientation. A planar map is made of
\emph{vertices}, \emph{edges} and \emph{faces}. The \emph{degree} of a
vertex or face is the number of edges incident to it (counted with
multiplicity).  Following \cite{BouSch00}, we
consider a particular class of planar maps called
\emph{constellations}, see Fig.\ref{fig:PathCons}.

\begin{Def}
  For $p \in \ldbrack 2, \infty \ldbrack$, a
  ($p$-)\emph{constellation} is a planar map whose faces are colored
  black or white in such a way that~:
\vspace{-0.5cm}
\begin{itemize}
\item adjacent faces have opposite colors,
  \item the degree of any black face is $p$,
  \item the degree of any white face is a multiple of $p$.
  \end{itemize}
\end{Def}

Each edge of a constellation receives a canonical orientation by
requiring that the white face is on its right.  It is easily seen that
the length of each oriented cycle is a multiple of $p$, and that any two
vertices are accessible from one another. A constellation is
\emph{rooted} if one of its edges is distinguished. The \emph{white
  root face} and \emph{black root face} are respectively the white and
black faces incident to the root edge. The \emph{root degree} is the
degree of the white root face. A constellation is said to be
\emph{pointed} if it has a distinguished vertex. Note that, in a
pointed rooted constellation, the pointed vertex is not necessarily
incident to the root edge.

In this paper, we consider $p$-constellations subject to a control on
white face degrees, i.e. for each positive integer $k$ we fix the
number of white faces of degree $kp$. This amounts to considering
multivariate generating functions of constellations depending on an
infinite sequence of variables $(x_k)_{k \geq 1}$, where $x_k$ is the
weight per white face of degree $kp$ and the global weight of a
constellation is the product of the weights of its white faces.  Two
families of constellation generating functions will be of interest
here.

The first family is that of rooted constellations with a prescribed
root degree. More precisely, for $n \geq 1$, let $F_n \equiv
F_n(x_1,x_2,\ldots)$ denote the generating function of rooted
$p$-constellations with root degree $n p$. By convention, we do not
attach a weight $x_n$ to the white foot face and we set $F_0 = 1$.

The second family is that of pointed rooted constellations with a
bounded ``distance'' between the root edge and the pointed
vertex. More precisely, in a pointed constellation, we say that a
vertex $v$ is \emph{of type} $j \in \naturals$ if $j$ is the minimal
length of an oriented path from $v$ to the pointed vertex (due to the
orientation, it is slightly inappropriate to think of $j$ as distance)
and we say that an edge is \emph{of type} $j \to j'$ if its origin and
endpoint are respectively of type $j$ and $j'$. Then, for $i \geq 1$,
let $V_i \equiv V_i(x_1,x_2,\ldots)$ denote the generating function of
pointed rooted $p$-constellations where the root edge is of type $j
\to j-1$ with $j \leq i$. Note that $V_1$ is the generating function
for rooted constellations (since for $i=1$, the pointed vertex
must be the endpoint of the root edge). Furthermore, we add a
conventional term $1$ to $V_i$, for all $i \geq 1$.

The fundamental observation of this paper may be stated as:
\begin{thm} \label{thm:multicont}
  The sequences $(F_n)_{n\geq 0}$ and $(V_i)_{i \geq 1}$ are related via
  the multicontinued fraction expansion
  \begin{equation}
    \label{eq:multicont}
    \sum_{n \geq 0} F_n t^n =
    \cfrac{1}{1 - t \displaystyle \prod_{i_1=1}^{p-1} \cfrac{V_{i_1}}{
        1 - t \displaystyle \prod_{i_2=1}^{p-1} \cfrac{V_{i_1+i_2}}{
          1 - t \displaystyle \prod_{i_3=1}^{p-1} \cfrac{V_{i_1+i_2+i_3}}{\ddots}}}}
    % F(t) := \sum_{n \geq 0} F_n t^n =
    % \cfrac{1}{1 - t
    %   \cfrac{V_1}{1 - t \cfrac{V_2}{\ddots} \cdot \cfrac{V_3}{\ddots}
    %     \cdots \cfrac{V_p}{\ddots}} \cdot
    %   \cfrac{V_2}{1 - t \cfrac{V_3}{\ddots} \cdot \cfrac{V_4}{\ddots}
    %     \cdots \cfrac{V_{p+1}}{\ddots}} 
    %   \cdots
    %   \cfrac{V_{p-1}}{1 - t \cfrac{V_p}{\ddots} \cdot \cfrac{V_{p+1}}{\ddots}
    %     \cdots \cfrac{V_{2p-2}}{\ddots}}
    % }
  \end{equation}
\end{thm}
In the case $p=2$, the r.h.s. of \eqref{eq:multicont} reduces to an
ordinary continued fraction (of Stieljes-type). This corresponds to
the bipartite case discussed in \cite[Eq. (1.13)]{BouGui10}: indeed,
$2$-constellations may be identified with bipartite planar maps upon
``collapsing'' the bivalent black faces into non-oriented edges.

\begin{figure}[t]
\begin{center}
\includegraphics{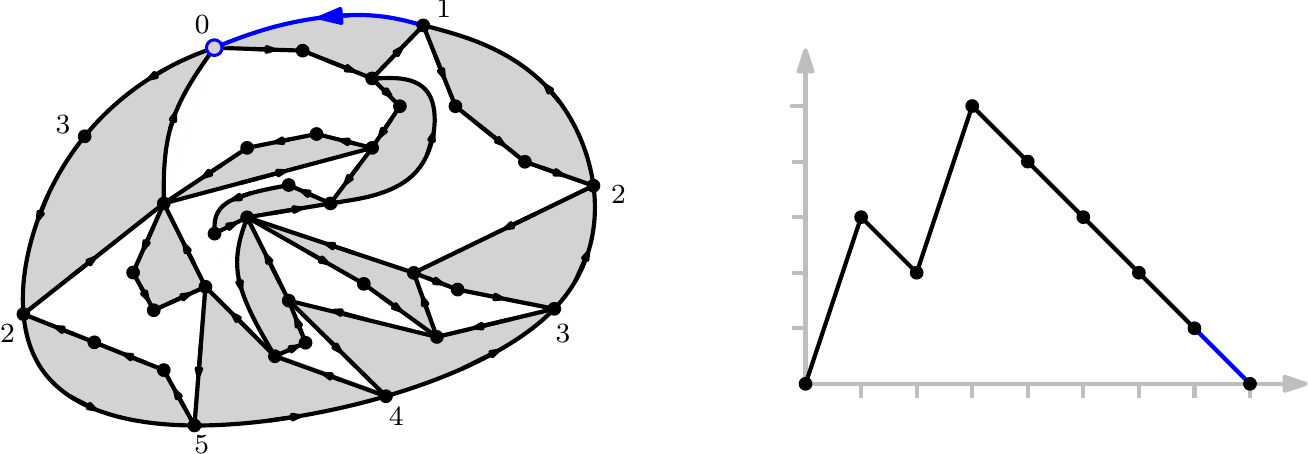}
\end{center}
\caption{A rooted 4-constellation of $F_2$, with weight $x_1^3x_2^4$ and the 4-excursion associated to it.}
\label{fig:PathCons}
\end{figure}

In the spirit of the combinatorial theory of continued fractions
initiated by \cite{FlaContFrac}, an alternate formulation of Theorem
\ref{thm:multicont} may be given in terms of some lattice paths. We
call \emph{$p$-path} a lattice path on $\integers \times \naturals$
made of two types of steps: \emph{rises} $(1,p-1)$ and \emph{falls}
$(1,-1)$. Note that a $p$-path starting from $(i_0,j_0)$ only visits
vertices $(i,j)$ with $i+j \equiv i_0+j_0 \mod p$.  A
\emph{$p$-excursion} of length $np$ is a $p$-path that starts at
$(0,0)$ and ends at $(np,0)$.  It is well-known that such
$p$-excursions are in one-to-one correspondence with $p$-ary rooted
plane trees with $n$ nodes, in number $\frac{1}{np+1}
\binom{np+1}{n}$. To each fall in a $p$-path, we attach a weight $V_i$
where $i$ is the starting height of the fall (i.e. the fall starts
from $(j,i)$ for some $j$, and thus ends at $(j+1,i-1)$). We define
the weight of a $p$-path as the product of the weights of its
falls. As discussed in Section \ref{sec:inversion}, the alternate formulation
of Theorem \ref{thm:multicont} is then:

\begin{thm} \label{thm:constpath} For all $n \geq 0$, $F_n$ is equal
  to the sum of weights of all $p$-excursions of length $np$.
\end{thm}

We prove this theorem in Section \ref{sec:constellations}.  As an
illustration, for $p=3$ and $n=1,2,3$, the equality reads
\begin{equation}
  \label{eq:firstfewfn}
  \begin{split}
    F_1=V_1 V_2&\\
    F_2=V_1 V_2&(V_1 V_2+V_2 V_3+V_3 V_4) \\
    F_3=V_1 V_2&(V_1^2 V_2^2 + 2 V_1 V_2^2 V_3 + V_2^2 V_3^2 + 2 V_1
    V_2 V_3 V_4 + 2 V_2 V_3^2 V_4 + \\ & \phantom{(}V_3^2 V_4^2 + V_2 V_3 V_4 V_5
    + V_3 V_4^2 V_5 + V_3 V_4 V_5 V_6)
  \end{split}
\end{equation}

We are now interested in inverting the relation \eqref{eq:multicont},
i.e. expressing $V_i$ in terms of the $F_n$'s. For $p=2$, this may be
done using Hankel determinants (see below). However, as soon as $p
\geq 3$, it is not difficult to see that knowing the sequence
$(F_n)_{n \geq 0}$ alone is not sufficient (for instance in
\eqref{eq:firstfewfn} it appears that we have ``twice as many unknowns
as equations'').

We thus need some extra knowledge. A possible way is to consider, for
all $n \geq 0$ and $r \geq 0$, the sum of weights of all $p$-paths
that start at $(-r,r)$ and end at $(np,0)$, which we denote by
$F_n^{(r)}$. Note that $F_n^{(0)}=F_n$ and, since a $p$-excursion
necessarily starts with a rise, $F_n^{(p-1)}=F_{n+1}$. From now on, we
restrict the values of $r$ to the interval $\ldbrack 0,p-2 \rdbrack$.
As will be discussed in Section \ref{sec:constellations}, those
$F_n^{(r)}$ have a natural interpretation in terms of constellations.
Let us write down the first few $F_n^{(1)}$ for $p=3$:
\begin{equation}
  \label{eq:firstfewfn1}
  \begin{split}
    F_0^{(1)}&=V_1 \\
    F_1^{(1)}&=V_1 V_2 (V_1 + V_3) \\
    F_2^{(1)}&=V_1 V_2 (V_1^2 V_2 + 2 V_1 V_2 V_3 + V_2 V_3^2 +
    V_1 V_3 V_4 + V_3^2 V_4 + V_3 V_4 V_5) \\
  \end{split}
\end{equation}
It appears that, interlacing the equations in \eqref{eq:firstfewfn1} and
\eqref{eq:firstfewfn}, we obtain a triangular system of equations for
$V_1,..,V_6$. This fact turns out to be general and, furthermore, the
solution to our inverse problem is provided by the following formula,
which we derive combinatorially in Section \ref{sec:inversion}.

\begin{thm} \label{thm:det}
  For $m \in \ldbrack 0,p-1 \rdbrack$ and $n \in \naturals$, we have the
  determinantal identity
  \begin{equation}
    \label{eq:det}
    \HH p m n := \det_{0 \leq i,j \leq n} F_{q_{i+m} + j}^{(r_{i+m})} = 
    \prod_{i=0}^n \prod_{j=1}^{ip+m} V_j
  \end{equation}
  where, for $k \in \naturals$, $q_k$ and $r_k$ denote respectively
  the quotient and the remainder in the Euclidean division of $k$ by
  $p-1$, namely $q_k=\lfloor \frac{k}{p-1} \rfloor$ and $r_k = k - (p-1) q_k$.
\end{thm}

\begin{cor}\label{cor:Vi}
   For $m \in \ldbrack 0,p-1 \rdbrack$ and $n \in \naturals$, we have
  \begin{equation}
    \label{eq:detcor}
    V_{pn+m} = 
    \begin{cases} 
      \dfrac{\HH p m n \HH p {m-1}{n-1}}{\HH p m{n-1} \HH p {m-1}n
      }&\text{if }m\geq1
      \\
      \dfrac{\HH p 0 n \HH p {p-1}{n-2}}{\HH p 0{n-1} \HH p {p-1}{n-1}
      }&\text{if }m=0 \text{ and }n\geq 1.
    \end{cases}
  \end{equation}
  with the convention $\HH p m {-1}=1$.
\end{cor}

Note that, for $p=2$, we recover the Hankel determinants: $\HH 2 m n =
\det_{0 \leq i,j \leq n} F_{i+j+m}$ with $m=0$ or $1$. For $p=3$, the
first few determinants read
\begin{equation}
  \label{eq:firstfewdets}
  \begin{aligned}
    \HH 3 0 0 &= F_0^{(0)} = 1
    & \HH 3 1 0 &= F_0^{(1)} %= V_1
    & \HH 3 2 0 &= F_2^{(0)} %= V_1 V_2
    \\
    \HH 3 0 1 &=
    \begin{array}{|cc|}
      F^{(0)}_0 & F^{(0)}_1 \\
      F^{(1)}_0 & F^{(1)}_1
    \end{array}
    %= V_1 V_2 V_3
    & \HH 3 1 1 &=
    \begin{array}{|cc|}
      F^{(1)}_0 & F^{(1)}_1 \\
      F^{(0)}_1 & F^{(0)}_2
    \end{array} 
    %= V_1^2 V_2 V_3 V_4
    & \HH 3 2 1 &=
    \begin{array}{|cc|}
      F^{(0)}_1 & F^{(0)}_2 \\
      F^{(1)}_1 & F^{(1)}_2
    \end{array}
    %= V_1^2 V_2^2 V_3 V_4 V_5
    \\
    \HH 3 0 2 &=
    \begin{array}{|ccc|}
      F^{(0)}_0 & F^{(0)}_1 & F^{(0)}_2 \\
      F^{(1)}_0 & F^{(1)}_1 & F^{(1)}_2 \\
      F^{(0)}_1 & F^{(0)}_2 & F^{(0)}_3 \\
    \end{array} & & & &\\
  \end{aligned}
\end{equation}
and the reader may check, using \eqref{eq:firstfewfn} and
\eqref{eq:firstfewfn1}, that those determinants are indeed monomials
in the $V_i$'s, in agreement with \eqref{eq:det}. We conjecture that
there is a simple relation between $\HH p m n$ and
$u_{(n+1)p+m}^{(p-2)}$ as defined at Equation (6.31) in
\cite{diFraIntegrable}, so that Equation (6.30) \emph{ibid.}\ amounts to
Corollary \ref{cor:Vi}.

In this paper, we present a simple application of Theorem
\ref{thm:det} to the case of \emph{Eulerian triangulations}, namely
3-constellations where all the white faces are also triangular.  This
corresponds to specializing the generating functions defined above at
the values $p=3$, $x_1=x$, $x_k=0$ for $k\geq 2$, where $x$ is a
variable controlling the number of triangles. In this case, the
sequence $(V_i)_{i \geq 1}$ is known to satisfy the simple recurrence
\begin{equation}
  \label{eq:recVntri}
  V_i = 1 + x V_i \left( V_{i-1} + V_{i+1} \right), \qquad i \geq 1,\ V_0 = 0,
\end{equation}
see \cite{BDG03b}. Clearly, this equation fully determines $V_i$ as a
power series in $x$. In the same reference, it was observed that the
solution of this equation has a remarkably simple form
\begin{equation}
  \label{eq:triangalt}
  V_i = V \frac{(1-y^i)(1-y^{i+4})}{(1-y^{i+1})(1-y^{i+3})}.
\end{equation}
where $V$ and $y$ are power series determined by the equations $V=1+2xV^2$,
$y+y^{-1} = (x V)^{-1} - 2$. So far, there was no combinatorial
explanation for the form of Equation~\ref{eq:triangalt}. Theorem \ref{thm:det}
provides such an explanation. Let us introduce the Fibonacci
polynomials defined as follows \cite[eq. (62),
p.327]{AnalyticCombinatorics}
\begin{equation}
  \label{eq:fibdef}
  \varphi_{n+2}(z) = \varphi_{n+1}(z) - z \varphi_n(z), \qquad
  \varphi_0(z)=0, \qquad \varphi_1(z)=1,
\end{equation}
which are reciprocals of Chebyshev polynomials of the second kind,
i.e. they satisfy the relation
\begin{equation}
  \label{eq:fiby}
  \varphi_n \left( \frac{1}{y+y^{-1}+2} \right) = \frac{1-y^n}{(1-y)(1+y)^{n-1}}.
\end{equation}
Then, Equation~\ref{eq:triangalt} results from the following proposition,
which we prove in Section \ref{sec:eulertri}.

\begin{prop}\label{prop:det3}
  Let $p=3$, $x_k=x \delta_{k,1}$ for $k \geq 1$ and $V=\sum_{n=0}^{\infty}
  \frac{(2n)!}{n! (n+1)!} (2x)^n$. For all $n \geq 0$, we have
  \begin{equation}
    \label{eq:tridet}
    \begin{split}
      \HH 3 0 {n-1} &= V^{n(3n-1)/2} \varphi_{3n+1}(xV) \\
      \HH 3 1 {n-1} &= V^{n(3n+1)/2} \varphi_{3n+2}(xV) \\
      (1-xV) \HH 3 2 {n-1} &= V^{n(3n+3)/2} \varphi_{3n+3}(xV).
    \end{split}
  \end{equation}
  % \begin{equation}
  %   \label{eq:triang}
  %   V_n = V \frac{\varphi_{n}(x V) \varphi_{n+4}(x V)}{\varphi_{n+1}(x V) \varphi_{n+3}(x V)}.
  % \end{equation}
\end{prop}

\vspace{0.5cm}

\section{$p$-paths and multicontinued fractions}
\label{sec:inversion}

We consider multivariate generating functions for $p$-paths with, for
all $i\geq 1$, a weight $V_i$ per fall starting from a height $i$.
All results stated in this section will hold with $(V_i)_{i \geq 1}$ a
sequence of formal variables (i.e. its definition of Section
\ref{sec:mainresults} in terms of constellations is not
needed). Recall that, for $n,r \geq 0$, $F_n^{(r)}$ is defined as the
generating function of $p$-paths starting from $(-r,r)$ and ending at
$(np,0)$. Let us add an extra weight $t$ per rise and sum over all
lengths, to obtain the generating functions
\begin{equation}
  \label{eq:ftdef}
  F^{(r)}(t;V_1,V_2,\ldots) := \sum_{n \geq 0} F_n^{(r)} t^n.
\end{equation}

By elementary recursive decompositions of $p$-paths, we obtain the
recursive equations
\begin{equation}
  \label{eq:ftrel2}
  F^{(r)}(t;V_1,V_2,\ldots) = 
  \begin{cases}
    1 + t\, F^{(p-1)}(t;V_1,V_2,\ldots) & \text{if $r=0$} \\
    V_r\, F^{(0)}(t;V_{r+1},V_{r+2},\ldots) \, F^{(r-1)}(t;V_1,V_2,\ldots)
    & \text{if $r \geq 1$}.
  \end{cases}
\end{equation}
(For $r=0$ we remove the first rise. For $r \geq 1$ we perform
first-passage decomposition at height $r-1$.) We easily deduce
that
\begin{equation}
  \label{eq:fn0close}
  F^{(0)}(t;V_1,V_2,\ldots) = \frac{1}{
    1 - t \prod_{i=1}^{p-1} V_i F_n^{(0)}(t;V_{i+1},V_{i+2},\ldots)}
\end{equation}
By iterating this relation, we find that $F^{(0)}(t;V_1,V_2,\ldots)$
is equal to the multicontinued fraction on the r.h.s. of
\eqref{eq:multicont}. This shows that Theorem \ref{thm:constpath}
(established in Section \ref{sec:constellations}) implies Theorem
\ref{thm:multicont}.

\begin{figure}[htbp]
  \centering
  \includegraphics[width=\textwidth]{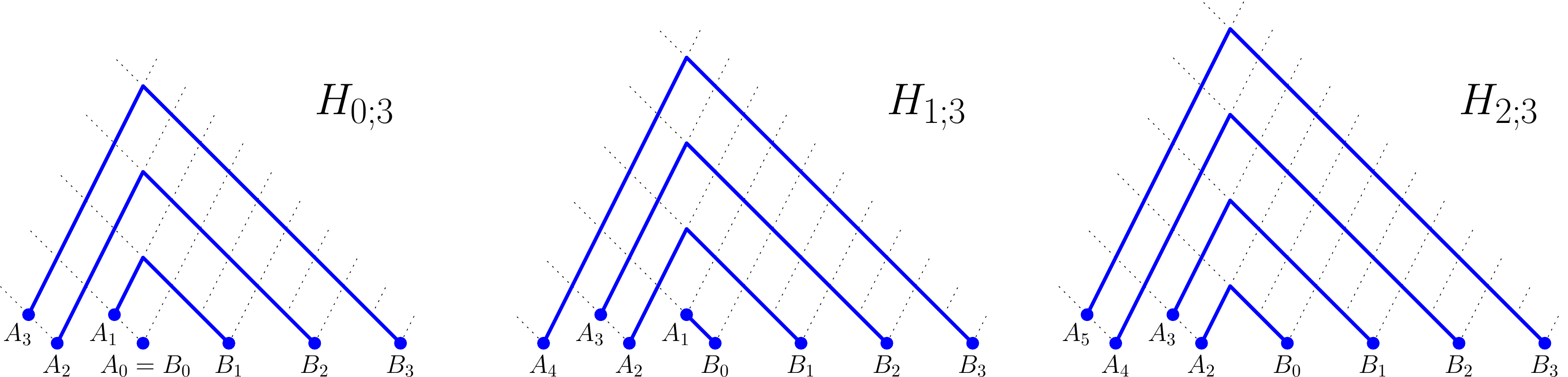}
  \caption{The unique configurations of non-intersecting lattice paths
    contributing respectively to $\HH 3 0 3$, $\HH 3 1 3$ and $\HH 3 2
    3$ for $p=3$.}
  \label{fig:Inversion2}
\end{figure}

We now turn to the proof of Theorem \ref{thm:det}, which is an
application of the celebrated LGV
lemma~\cite[]{Lindstrom,GesVie89}. We consider the weighted acyclic
directed planar graph whose vertices are the $(k,l) \in \integers
\times \naturals$ such that $k+l \equiv 0 \mod p $, and whose edges
are the rises $(k,l) \to (k+1,l+p-1)$, weighted $1$, and the falls
$(k,l) \to (k+1,l-1)$, weighted $V_l$. For $n \in \naturals$, let
$q_n$ and $r_n$ be respectively the quotient and the remainder in the
Euclidean division of $k$ by $p-1$, and let
\begin{equation}
  \label{eq:soursink}
  A_n := (-n - q_n,r_n) = (- p q_n - r_n, r_n), \qquad B_n:=(np,0).
\end{equation}
The generating function for paths from $A_i$ to $B_j$ ($i,j \geq 0$)
is nothing but $F_{q_i+j}^{(r_i)}$.

Theorem \ref{thm:det} then results from the LGV lemma and the
following proposition, illustrated on Fig.\ref{fig:Inversion2}.
\begin{prop}
  For $m \in \ldbrack 0,p-1 \rdbrack$, there is a unique configuration
  of non-intersecting lattice paths connecting the \emph{sources}
  $A_m,A_{m+1},\ldots,A_{m+n}$ to the \emph{sinks} $B_0,B_1,\ldots,B_n$: for
  $i \in \ldbrack 0,n \rdbrack$, the source $A_{m+i}$ is connected to
  the source $B_i$ via the highest possible path, passing through
  $(-m,m+ip)$. The weight of this configuration is $\prod_{i=0}^n
  \prod_{j=1}^{i p + m} V_j$.
\end{prop}
The proof is left to the reader. For $p=2$ we recover the known
combinatorial interpretation of Hankel determinants, see e.g.
\cite{ViennotLACIM}.

\section{Constellations, $p$-paths, and the slice decomposition}
\label{sec:constellations}

In this section, we establish Theorem~\ref{thm:constpath} and related
results.

Let us start with some definitions and notations. For $n \geq 1$, let
$\mathcal{F}_n$ (resp. $\mathcal{F}^\bullet_n$) be the set of rooted
(resp. pointed rooted) $p$-constellations with root degree $n p$. The
generating function of $\mathcal{F}_n$ is $F_n$ as defined in Section
\ref{sec:mainresults}. Similarly, let $\mathcal{V}_i$ ($i \geq 1$) be
the set of pointed rooted $p$-constellations whose root edge is of
type $j \to j-1$ with $j \leq i$, to which we add a conventional
``empty map'' with weight $1$. The generating function of
$\mathcal{V}_i$ is $V_i$.

A \emph{constellated} path is a $p$-path such that, for all $i \geq
1$, an element of $\mathcal{V}_i$ is associated with each fall from
height $i$. The weight of a constellated path is the product of the
weights of its associated constellations. We shall consider
constellated excursions and bridges, where a \emph{$p$-bridge} of
length $np$ is a $p$-path starting from $(0,j)$ and ending at $(np,j)$
from some $j \in \naturals$. Theorem~\ref{thm:constpath} then follows from:
\begin{prop}\label{prop:slice}

  There is a weight-preserving bijection between $\mathcal{F}_n$ and
  the set of constellated excursions of length $np$. More
  generally, there is a weight-preserving bijection between
  $\mathcal{F}^\bullet_n \times \naturals$ and the set of constellated
  bridges of length $np$.
\end{prop}

This bijection is a byproduct of the correspondence between
constellations and mobiles introduced in \cite{BDG04}. Here we provide
a direct construction, the so-called \emph{slice decomposition}, extending
the one given for $p=2$ by \cite{BouGui10}.

\begin{figure}[t]
\centering
\includegraphics{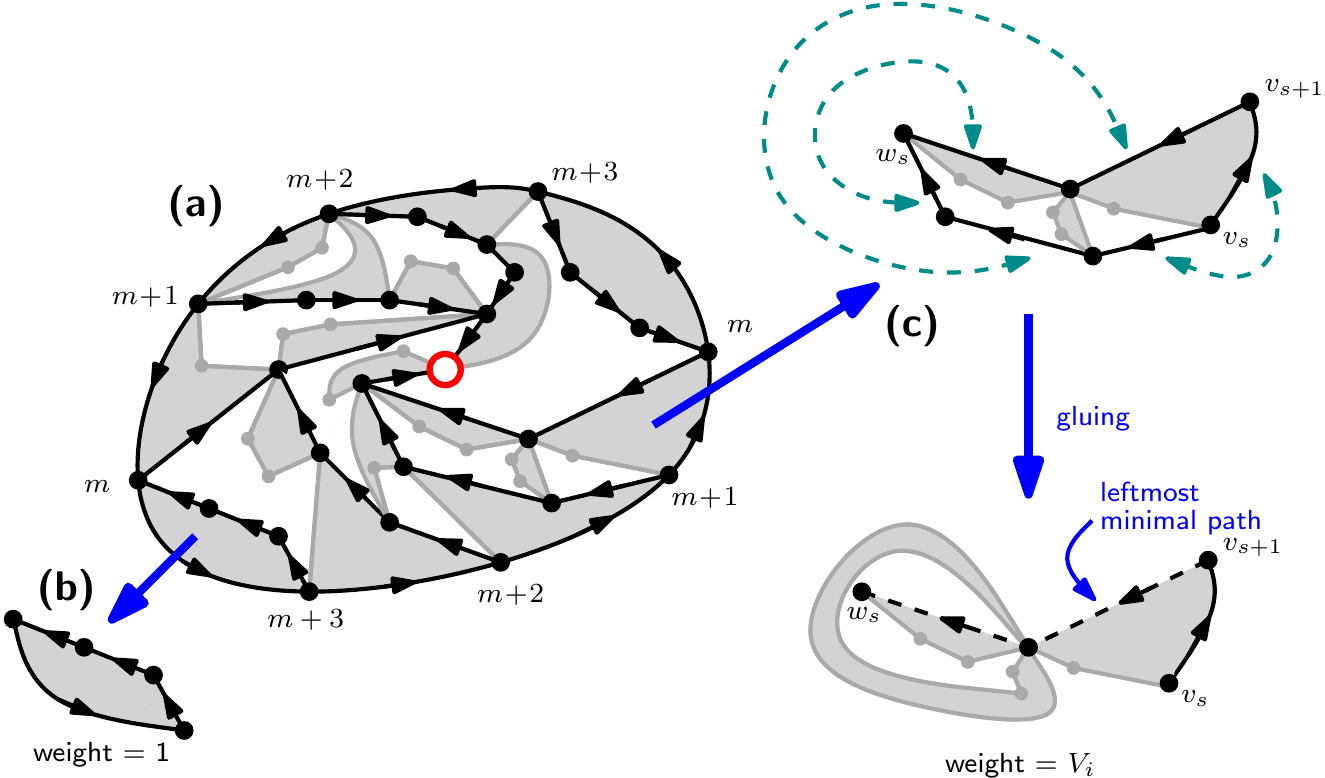}
\caption{\label{fig:slices}
The decomposition of a rooted pointed $4$-constellation along leftmost
minimal path into slices (a). Slice corresponding to edge $(m,m+3)$ are reduced
to a black edge (b). Identifying the two sides of the boundary of a slice of an
edge $(m,m-1)$ yields to a
pointed-rooted constellation (c)}
\end{figure}

\newcommand\tpv{the pointed vertex}

We start from a constellation $C \in \mathcal{F}^\bullet_n$, and a
$\ell \in \naturals$ (whose role will be discussed at the very
end). The white root face of $C$ forms a directed cycle which we
denote by $(v_0,v_1,\ldots,v_{np-1},v_{np}=v_0)$ where, say, $v_0$ is
the endpoint of the root edge (hence $v_{np-1}$ is its origin).

Let $j_s$ be the type of $v_s$ (recall that the type of a vertex $v$
is the minimal length of an oriented path from $v$ to \tpv), and let
\begin{equation}
  \label{eq:P}
  P := ((0,j_0+\ell),(1,j_1+\ell),\ldots,(np,j_{np}+\ell)).
\end{equation}
We claim that $P$ is a $p$-bridge i.e. for all $s \in \ldbrack 0,np-1
\rdbrack$,
\begin{equation}
  \label{eq:typevar}
  j_{s+1} - j_s \in \{ -1 , p-1 \}.
\end{equation}
Indeed, since $(v_s,v_{s+1})$ is an oriented edge, we have $j_s \leq
j_{s+1}+1$ by the definition of type. Furthermore, the black face
incident to $(v_s,v_{s+1})$ forms an oriented path of length $p-1$
from $v_{s+1}$ to $v_s$, hence $j_{s+1} \leq j_s+p-1$. Finally, since
$p$ divides the length of any oriented cycle on $f$, we have $j_s
\equiv j_{s+1} + 1 \mod p$, which concludes the proof of our claim. In
particular, when $C$ is a rooted constellation that we point at the
endpoint of the root edge, and $\ell=0$, then $P$ is a $p$-excursion.

Let us now explain how $P$ is constellated. Assume that the
constellation $C$ is embedded in the plane with the white root face as
outer face, and consider, for each $s$, the \emph{leftmost} minimal
(oriented) path from $v_s$ to \tpv, where minimal refers to the length of the
path. This family of paths leads to a
decomposition of $C$ into connected components that we call
\emph{slices}, see Fig.\ref{fig:slices}(a). A slice is associated
with each edge $(v_s,v_{s+1})$ incident to the white root face (thus
with each step of $P$). More precisely, it is delimited by
$(v_s,v_{s+1})$ and the two paths starting from $v_s$ and $v_{s+1}$
(in general, those two paths merge before reaching \tpv, and we
remove their common part).

Observe that when $j_{s+1}=j_s+p-1$, the slice is reduced to a single
black face (as the path starting from $v_{s+1}$ passes through $v_s$
after circumventing this face) with weight $1$, see
Fig.~\ref{fig:slices} (b).

Let us now assume that $j_{s+1}=j_s-1$: we claim that the slice
corresponds to an element of $\mathcal{V}_{j_s+\ell}$. Two cases may
occur. If the path starting from $v_s$ passes through $v_{s+1}$, then
the slice is empty, and corresponds to the empty map in
$\mathcal{V}_{j_s+\ell}$.  Otherwise, let $w_s$ be the vertex where
the two paths starting from $v_s$ and $v_{s+1}$ merge. The boundary of
the slice is therefore made of two (non-empty) oriented paths from
$v_s$ to $w_s$ that do not meet except at their extremities. By
construction those two paths have the same length: we may identify
pairwise their edges, see Fig.\ref{fig:slices}(c). This
identification preserve the degrees of the white faces, hence the
weights, and the orientations guarantee that the resulting map is a
constellation $C_s$, which we root at $(v_s,v_{s+1})$ and point at
$w_s$. By construction, the identified paths form in $C_s$ the
leftmost minimal path starting with the root edge and ending at
\tpv. Thus, in $C_s$ the root edge is of type $j'_s \to j'_s-1$, with
$j'_s \leq j_s \leq j_s+\ell$, so that $C_s \in
\mathcal{V}_{j_s+\ell}$ as claimed. Note that since the pointed vertex
of $C$ is incident to at least one white face, there is at least one
$s$ such that $j'_s=j_s$ i.e. $C_s \in \mathcal{V}_{j_s} \setminus
\mathcal{V}_{j_s-1}$. This remark allows to characterize $\ell$ as the
largest integer such that $P - (0,\ell)$ is still a constellated
bridge.

In conclusion, the slice decomposition yields a constellated
bridge. By following the steps in reverse order, it is clear how to
reconstruct a pointed rooted constellation from a constellated
bridge. By the above remark we also recover the integer $\ell$ and we may check
that the correspondence is one-to-one (note that
constellated bridges differing by a height shift yield the same
constellation, but different values of $\ell$). 
Thereby we prove
Proposition~\ref{prop:slice} and Theorem~\ref{thm:constpath}.

Let us now mention other byproducts of our construction.

\begin{prop} \label{prop:intFni} For all $n \geq 0$ and $r \in
  \ldbrack 0,p-1 \rdbrack$, $F_n^{(r)}$ is the generating function of
  rooted $p$-constellations with root degree $(n+1) p$ such that, if
  we denote by $(v_0,v_1,\ldots,v_{(n+1)p}=v_0)$ the directed cycle
  corresponding to the white root face ($v_0$ being the endpoint of
  the root edge), then all the vertices $v_1,v_2,\ldots,v_{p-1-r}$ are
  bivalent.
\end{prop} 
\begin{proof}
  Apply the slice decomposition with $v_0$ as pointed vertex and
  $\ell=0$. The leftmost minimal path from $v_1$ to $v_0$ visits
  successively $v_2,\ldots,v_{p-r}$, so that the constellated bridge
  starts with one rise followed by $p-1-r$ falls corresponding to
  empty slices. Removing those $p-r$ trivial steps, we obtain a
  constellated path with $np+r$ steps going from height $r$ to height
  $0$.
\end{proof}

\begin{prop}[\cite{BDG04,diFraIntegrable}]
  Let $F_n^{(i-1;i)}$ be the generating function for constellated
  paths from $(0,i-1)$ to $(np-1,i)$. We have:
  \begin{equation} \label{eq:recVn} V_i = 1 + V_i \sum_{n=1}^\infty x_n
    F_{n}^{(i-1;i)}
  \end{equation}
\end{prop}
\begin{proof}%[(new version)]
  Apply the slice decomposition to a (non-empty) map in
  $\mathcal{V}_i$. Taking $\ell=i-j$, where $j \to j-1$ is the type of
  the root edge, we obtain a constellated bridge of arbitrary length,
  starting at height $i-1$ and ending with a fall from height $i$,
  thus the factor $V_i$. The weight $x_n$ accounts for the white
  root face.
\end{proof}

Note the recursive nature of \eqref{eq:recVn}: $F_{n}^{(i-1;i)}$ is a
polynomial in the $V_j$'s such that $|j-i|<(p-1)n$. In particular, for
$i \to \infty$, $V_i$ converges (in a suitable topology) to $V$ the
generating function of $\mathcal{V}=\bigcup_{i\geq 1} \mathcal{V}_i$,
and $F_{n}^{(i-1;i)} \to \binom{np-1}{n} V^{n(p-1)-1}$. Hence $V$
satisfies the recursive equation
\begin{equation}\label{eq:recV}
  V = 1 + \sum_{k=1}^\infty \binom{np-1}{n} x_n V^{n(p-1)}.
\end{equation}
By the Lagrange inversion formula, we may obtain an explicit
expression for the coefficients of $V$ \cite[p.131]{BouPhD}, and
recover the number of rooted constellations with a prescribed degree
distribution \cite[Theorem 2.3]{BouSch00}.

By extending the construction in \cite[Section 3.3]{BouGui10}, we may
express the generating function $F_n$ in terms of $V$ as follows.

\begin{prop}\label{prop:consquant}
  For $n\geq 0$, and $V$ the solution of \eqref{eq:recV}, we have
  \begin{equation}\label{eq:Fpk}
    F_{n} = \frac{1}{np+1}\binom{np+1}{n}V^{n(p-1)+1}-\sum_{k \geq 1}\left(\sum_{j =
        0} \frac{jp+1}{np+1}\binom{np+1}{n-j}\binom{kp-1}{k+j}\right)x_kV^{(k+n)(p-1)}.
  \end{equation}
\end{prop}

It would be interesting to have a similar formula for all $F_n^{(r)}$.
We could then apply Theorem~\ref{thm:det} and Corollary~\ref{cor:Vi} and
deduce a general expression for $V_i$. We may express $F_n^{(r)}$ in
terms of $F_n$ using a Tutte-like decomposition \cite[Section
5.3]{BouJeh06}, for instance for $p=3$ and $r=1$ we have

\begin{equation}
  F^{(1)}_n = \sum_{l \geq 1} x_l F_{n+l} + \sum_{i=0}^{n} F_i F_{n-i}.
  \label{eq:fk1tri}
\end{equation}

Applying Proposition~\ref{prop:consquant} yields in principle to an
expression of $F_n^{(r)}$ in terms of $V$. In practice it seems to
quickly become intractable, except in the case of Eulerian
triangulations on which we focus next.

\section{Application to Eulerian triangulations}
\label{sec:eulertri}

We now specialize our formulas to the case of Eulerian triangulations,
i.e. $p=3$ and $x_k=x \delta_{k,1}$. Note first that \eqref{eq:recVn}
reduces to \eqref{eq:recVntri}, while \eqref{eq:recV} yields the simple
equation $V=1+2xV^2$, hence $V$ is as in
Proposition \ref{prop:det3}. Furthermore, \eqref{eq:Fpk} and
\eqref{eq:fk1tri} become respectively:
\begin{equation}\label{eq:FketF1tri}
  F_n = \left( p^{(0)}_n \cdot (1-2xV) -p^{(3)}_{n-1} \cdot (xV) \right) V^{2n+1},
  \qquad
  F_n^{(1)}= x F_{n+1} + \sum_{i=0}^{n} F_i F_{n-i},
\end{equation}
where we introduce the shorthand notation $p_n^{(r)} :=
\frac{r+1}{3n+r+1} \binom{3n+r+1}{n}$, the number of $3$-paths from
$(-r,r)$ to $(0,3n)$. By the relation
$p_n^{(1)}=p_n^{(0)}+p_{n-1}^{(3)}$ (whose path interpretation is
obvious), we may rewrite
\begin{equation}
  \label{eq:Fktribis}
  F_n = \left( p^{(0)}_n \cdot(1-xV) -p^{(1)}_{n}\cdot(xV) \right) V^{2n+1}.
\end{equation}
Using this relation in the expression \eqref{eq:FketF1tri} for
$F_n^{(1)}$, and using some summation formulas for the $p_n^{(r)}$'s
(following from several classical path decompositions), we arrive at
\begin{equation}
  \label{eq:Fk1tribis}
   F_n^{(1)} = \left( p^{(1)}_n \cdot(1-xV)^2 -p^{(0)}_{n+1}\cdot(xV) \right) V^{2n+2}.
\end{equation}

We are now ready to substitute these formulas into the determinants
$\HH 3 n m$ ($m=0,1,2$) so as to prove Proposition~\ref{prop:det3}. We
will once again use the LGV lemma, to give another non-intersecting
lattice path interpretation (\emph{NILP}) of the $\HH 3 n m$'s. Let us
consider the same acyclic directed planar graph introduced in Section
\ref{sec:inversion} for $p=3$ (i.e. the graph on which the $3$-paths
live), but now with a uniform weight $1$ on every edge (rise or
fall). Let $A_n$ and $B_n$ be as in \eqref{eq:soursink} (for $p=3$),
namely $A_n=(-\lfloor 3n/2 \rfloor, r_n)$ and $B_n=(3n,0)$, and let us
extend the graph by adding the vertices
\begin{equation}
  \label{eq:newsources}
  A'_n := (-(3n+1)/2,-1), \qquad n \geq 0
\end{equation}
and the special weighted edges
\begin{equation}
  \begin{cases}
    e_{2n}:=A'_{2n} \to A_{2n} &\text{weighted }1\!-\!xV\\
    \bar e_{2n}:=A'_{2n} \to A_{2n+1} &\text{weighted }-\!xV
  \end{cases}
  \quad
  \begin{cases}
    e_{2n+1}:=A'_{2n+1} \to A_{2n+1} &\text{weighted }(1\!-\!xV)^2\\
    \bar e_{2n+1}:=A'_{2n+1} \to A_{2n+2} &\text{weighted }-\!xV
  \end{cases}.
  \label{eq:specialedges}
\end{equation}
We readily see that the resulting graph is planar and,
up to a factor $V^{2(i+j)+1}$ (resp.  $V^{2(i+j+1)}$), that the
generating function for paths from $A'_{2i}$ (resp. $A'_{2i+1}$) to
$B_j$ ($i,j \geq 0$) is nothing but $F_{i+j}$ (resp. $F_{i+j}^{(1)}$).
We may get rid of those extra factors by defining
\begin{equation}
  T_{3k+1} := \frac{\HH 3 0 {k-1}}{V^{k(3k-1)/2}}, \qquad T_{3k+2} :=
  \frac{\HH 3 1 {k-1}}{V^{k(3k+1)/2}}, \qquad
  T_{3k+3} := \frac{\HH 3 2 {k-1}(1-xV)}{V^{k(3k+3)/2}}.\label{eq:Tdef}
\end{equation}
and it follows from the LGV lemma that $T_{3k+1}$ (resp.\ $T_{3k+2}$
and $T_{3k+3}/(1-xV)$) is the generating function of NILP
configurations from the sources $A'_0,\dots,A'_{k-1}$ (resp.\
$A'_1,\dots,A'_{k}$ and $A'_2,\dots,A'_{k+1}$) to the sinks
$B_0,\ldots,B_{k-1}$. By convention, empty configurations (for $n=0$)
receive the weight $1$ so that $T_1=T_2=1$ and $T_3=1-xV$. We claim
that the $T_n$'s satisfy the recurrence relation
\begin{equation}
  T_{n+3}  = (1-xV) T_{n+1} - xV T_{n} \label{eq:Tnrec}
\end{equation}
and this is sufficient to prove that $T_n=\varphi_n(xV)$, since the
Fibonacci polynomials with $z=xV$ satisfy the same recurrence and
initial conditions, by \eqref{eq:fibdef}. We prove \eqref{eq:Tnrec}
via a case-by-case analysis, depending on the residue of $n$ modulo 3.
Here we only check the case $n=3k+1$, and leave the others cases to
the reader (beware the extra $1-xV$ factor in $T_{3k+3}$). Consider a
NILP configuration from $A'_0,\dots,A'_{k}$ to $B_0,\ldots,B_{k}$, as
counted by $T_{3k+4}$. We first easily see that, for all $i$, the path $i$
(from $A'_i$ to $B_i$) visits $(0,3i)$, and all the steps afterward are
falls. We now distinguish whether the step before $(0,3k)$ on the
uppermost path is a rise or a fall.
\begin{figure}
\begin{center}
  \subfigure[If the last step before $(0,3k)$ is a rise, the upper-most path is
deleted.\label{subfig:LGV1}]{\includegraphics[scale =
0.85]{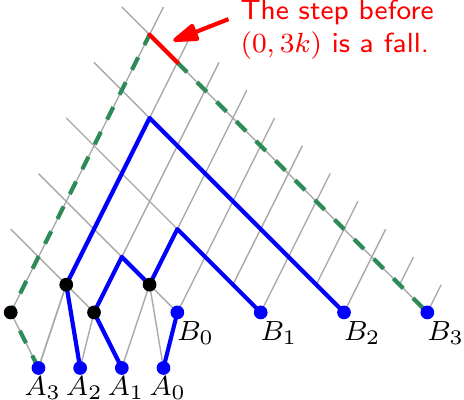}}\qquad\qquad
  \subfigure[If the last step before $(0,3k)$ is a fall, for each path, we
replace the steps after the vertex $(-1,3i-1)$ by $3i-1$ falls\label{subfig:LGV2}]{\includegraphics[scale=0.85]{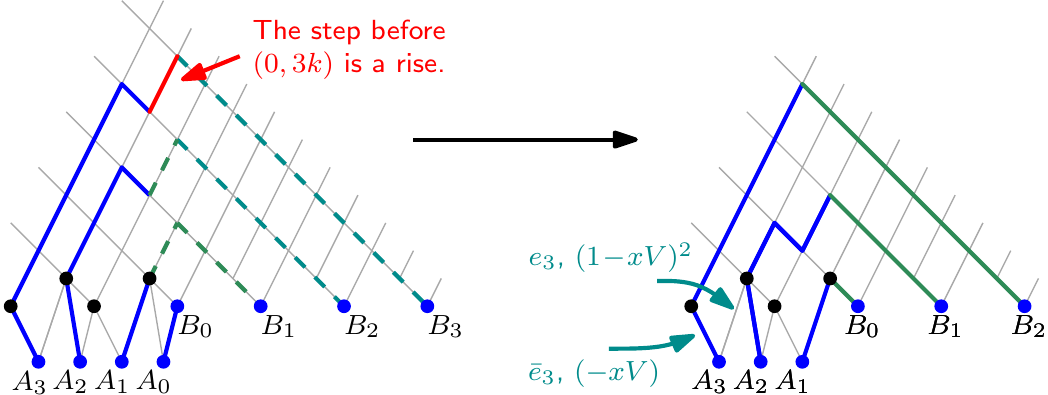}}
\caption{The decomposition of a configuration of $\HH 3 0 4$ into a
configuration of $\HH 3 0 3$ (see \subref{subfig:LGV1}) or of $\HH 3 1 3$ (see \subref{subfig:LGV2}).}
\end{center}
\label{fig:LGV3}
\end{figure}
\begin{itemize}
\item If it is a fall, then all the preceding steps are rises, so that
  the path visits $A_{k+1}$ and passes through the special edge $\bar
  e_{k}$, thus has weight $-xV$. All the remaining paths form an
  unconstrained NILP configuration from the sources
  $A'_0,\ldots,A'_{k-1}$ to the sinks $B_0,\ldots,B_{k-1}$, with
  generating function $T_{3k+1}$. The global contribution of this case
  is $-x V T_{3k+1}$ (see Fig.\ref{subfig:LGV1}).
\item If it is a rise, i.e. the uppermost path visits $(-1,3k-1)$,
  then, by non-intersection, the path $i$ must visit $(-1,3i-1)$ for
  all $i \geq 1$, and the path $0$ is reduced to the special edge
  $e_0$, thus has weight $1-xV$.  Remove the path $0$, and for each path
  $i \geq 1$, replace its steps after the vertex $(-1,3i-1)$ by $3i-1$ falls
  (see Fig.\ref{fig:LGV3}): the resulting configuration is a NILP
  configuration from the sources $A'_1,\ldots,A'_k$ to the sinks
   $B_0,\ldots,B_{k-1}$, with generating function $T_{3k+2}$.
  Since the correspondence is one-to-one, the global contribution of this
  case is $-x V T_{3k+2}$ (see Fig.\ref{subfig:LGV2}).
\end{itemize}
This shows that \eqref{eq:Tnrec} holds with $n=3k+1$. Together with the cases
$n=3k+2$ and $n=3k+3$, this concludes the proof of Proposition \ref{prop:det3}.

\section*{Acknowledgements}\label{sec:ack}
We thank P. Di Francesco for fruitful discussions. The work of M.A. in
partially funded by the ERC under the agreement "ERC StG 208471 -
ExploreMap".

\bibliographystyle{abbrvnat}
\bibliography{constellation}

\begin{thebibliography}{17}
\providecommand{\natexlab}[1]{#1}
\providecommand{\url}[1]{\texttt{#1}}
\expandafter\ifx\csname urlstyle\endcsname\relax
  \providecommand{\doi}[1]{doi: #1}\else
  \providecommand{\doi}{doi: \begingroup \urlstyle{rm}\Url}\fi

\bibitem[Bousquet-M{\'e}lou(2006)]{MBM06}
M.~Bousquet-M{\'e}lou.
\newblock Limit laws for embedded trees: applications to the integrated
  super{B}rownian excursion.
\newblock \emph{Random Structures and Algorithms}, 29\penalty0 (4):\penalty0
  475--523, 2006.
\newblock arXiv:math/0501266 [math.CO].

\bibitem[Bousquet-M{\'e}lou and Jehanne(2006)]{BouJeh06}
M.~Bousquet-M{\'e}lou and A.~Jehanne.
\newblock Polynomial equations with one catalytic variable, algebraic series
  and map enumeration.
\newblock \emph{Journal of Combinatorial Theory, Series B}, 96\penalty0
  (5):\penalty0 623--672, 2006.
\newblock arXiv:math/0504018 [math.CO].

\bibitem[Bousquet-M{\'e}lou and Schaeffer(2000)]{BouSch00}
M.~Bousquet-M{\'e}lou and G.~Schaeffer.
\newblock Enumeration of planar constellations.
\newblock \emph{Advances in Applied Mathematics}, 24\penalty0 (4):\penalty0
  337--368, 2000.

\bibitem[Bouttier(2005)]{BouPhD}
J.~Bouttier.
\newblock \emph{{Physique statistique des surfaces al{\'e}atoires et
  combinatoire bijective des cartes planaires}}.
\newblock These, Universit{\'e} Pierre et Marie Curie - Paris VI, June 2005.
\newblock URL \url{http://tel.archives-ouvertes.fr/tel-00010651/en/}.

\bibitem[Bouttier and Guitter(2011)]{BouGui10}
J.~Bouttier and E.~Guitter.
\newblock Planar maps and continued fractions.
\newblock \emph{Comm. Math. Phys.}, to appear, 2011.
\newblock \doi{10.1007/s00220-011-1401-z}.
\newblock arXiv:1007.0419 [math.CO].

\bibitem[Bouttier et~al.(2003{\natexlab{a}})Bouttier, Di~Francesco, and
  Guitter]{BDG03a}
J.~Bouttier, P.~Di~Francesco, and E.~Guitter.
\newblock Geodesic distance in planar graphs.
\newblock \emph{Nuclear Physics B}, 663\penalty0 (3):\penalty0 535 -- 567,
  2003{\natexlab{a}}.
\newblock arXiv:cond-mat/0303272 [cond-mat.stat-mech].

\bibitem[Bouttier et~al.(2003{\natexlab{b}})Bouttier, Di~Francesco, and
  Guitter]{BDG03b}
J.~Bouttier, P.~Di~Francesco, and E.~Guitter.
\newblock Statistics of planar graphs viewed from a vertex: a study via labeled
  trees.
\newblock \emph{Nuclear Physics B}, 675\penalty0 (3):\penalty0 631--660,
  2003{\natexlab{b}}.
\newblock arXiv:cond-mat/0307606 [cond-mat.stat-mech].

\bibitem[Bouttier et~al.(2004)Bouttier, Di~Francesco, and Guitter]{BDG04}
J.~Bouttier, P.~Di~Francesco, and E.~Guitter.
\newblock Planar maps as labeled mobiles.
\newblock \emph{Electron. J. Combin}, 11\penalty0 (1):\penalty0 R69, 2004.
\newblock arXiv:math/0405099 [math.CO].

\bibitem[Di~Francesco(2005)]{diFraIntegrable}
P.~Di~Francesco.
\newblock Geodesic distance in planar graphs: An integrable approach.
\newblock \emph{The Ramanujan Journal}, 10\penalty0 (2):\penalty0 153--186,
  2005.
\newblock arXiv:math/0506543 [math.CO].

\bibitem[Flajolet(1980)]{FlaContFrac}
P.~Flajolet.
\newblock Combinatorial aspects of continued fractions.
\newblock \emph{Discrete Mathematics}, 32\penalty0 (2):\penalty0 125--161,
  1980.

\bibitem[Flajolet and Sedgewick(2009)]{AnalyticCombinatorics}
P.~Flajolet and R.~Sedgewick.
\newblock \emph{Analytic combinatorics}.
\newblock Cambridge University Press, 2009.

\bibitem[Gessel and Viennot(1989)]{GesVie89}
I.~Gessel and X.~Viennot.
\newblock Determinants, paths, and plane partitions.
\newblock \emph{preprint}, 132\penalty0 (197.15), 1989.

\bibitem[Goulden and Jackson(1997)]{GouJack97}
I.~Goulden and D.~Jackson.
\newblock Transitive factorisations into transpositions and holomorphic
  mappings on the sphere.
\newblock \emph{Proceedings of the American Mathematical Society}, 125\penalty0
  (1):\penalty0 51--60, 1997.

\bibitem[Hurwitz(1891)]{Hurwitz}
A.~Hurwitz.
\newblock {\"U}ber {R}iemann'sche {F}l{\"a}chen mit gegebenen
  {V}erzweigungspunkten.
\newblock \emph{Mathematische Annalen}, 39\penalty0 (1):\penalty0 1--60, 1891.

\bibitem[Lando and Zvonkin(2004)]{LandoZvonkin}
S.~Lando and A.~Zvonkin.
\newblock \emph{Graphs on surfaces and their applications}, volume 141.
\newblock Springer Verlag, 2004.

\bibitem[Lindstr\"om(1969)]{Lindstrom}
B.~Lindstr\"om.
\newblock Determinants on semilattices.
\newblock \emph{Proc. Amer. Math. Soc.}, 20:\penalty0 207--208, 1969.

\bibitem[Viennot(1998)]{ViennotLACIM}
X.~G. Viennot.
\newblock Une th\'eorie combinatoire des polyn\^omes orthogonaux.
\newblock \emph{Notes de Cours}, UQAM, Montr\'eal, 1998.
\newblock URL
  \url{http://web.mac.com/xgviennot/Xavier_Viennot/polyn%C3%B4mes_orthogonaux.%
html}.

\end{thebibliography}

\end{document}